\documentclass[a4paper,10pt]{amsart}

\usepackage{amsfonts,amsthm,amssymb,amsmath,amscd}

\newtheorem{lem}{Lemma}[section]
\newtheorem{prop}{Proposition}[section] 
 
\newtheorem{rem}{Remark}[section] 
\newtheorem{tm}{Theorem}[section] 

\newcommand{\ZZ}{\mathbb{Z}}
\newcommand{\QQ}{\mathbb{Q}}

\newcommand{\CC}{\mathbb{C}}
\newcommand{\PP}{\mathbb{P}}
\newcommand{\FF}{\mathbb{F}}

\newcommand{\e}{\varepsilon}
\newcommand{\w}{\omega}
\newcommand{\la}{\lambda}

\newcommand{\HH}{\mathbb{H}}
\newcommand{\DD}{\mathbb{D}}

\newcommand{\GL}{\mathrm{GL}}
\newcommand{\SL}{\mathrm{SL}}

\newcommand{\HG}{\mathrm{H}\Gamma}

\newcommand{\I}{\mathrm{I}}
\newcommand{\IM}{\mathrm{Im}}
\newcommand{\RE}{\mathrm{Re}}

\newcommand{\OM}{\mathrm{O}^+(M)}
\newcommand{\SOM}{\mathrm{SO}^+(M)}
\newcommand{\OMK}{\mathrm{O}^+_{K3}(M)}
\newcommand{\SOMK}{\mathrm{SO}^+_{K3}(M)}
\newcommand{\OME}{\mathrm{O}^+_{Enr}(M)}
\newcommand{\SOME}{\mathrm{SO}^+_{Enr}(M)}
\begin{document}
\title{Moduli space of Hessian K3 surfaces and arithmetic quotients}
\date{}
\author{Kenji Koike, Yamanashi University}  
\maketitle
\begin{abstract}
Let $X$ be a hypersurface in $\PP^n$ given by $F(t_0, \cdots, t_n) = 0$. 
The Hessian hypersurface $H(X)$ of $X$ is defined by 
$\det (\partial^2 F / \partial t_i \partial t_j) = 0$.
If X is a cubic surface, then $H(X)$ is a quartic surface which is classically known 
as a symmetroid. The minimal desingularization of $H(X)$ is a K3 surface and it has an Enriques involution.  
We study the moduli space of them as arithmetic quotients.
\end{abstract}
\section{Hessian K3 surfaces of cubic surfaces} 
We first review classical facts and recent results on cubic surfaces and Hessian K3 surfaces.
\subsection{The Sylvester pentahedral form and classical invariants}
For quaternary cubic forms 
\[
 \sum_{i+j+k+l=3} a_{ijkl}X^i Y^j Z^k W^l \in \CC[X,Y,Z,W],
\]
the ring of $\SL_4(\CC)$-invariants $\CC[a_{ijkl}]^{\SL_4(\CC)}$ is 
\[
  \CC[I_8, I_{16}, I_{24}, I_{32}, I_{40}, I_{100}]
\] 
where $I_n$ is an invariant polynomial of degree $n$ (\cite{Sa}, \cite{Hu}). We have 
$I_{100}^2 \in \CC[I_8, I_{16}, I_{24}, I_{32}, I_{40}]$ and $I_8, \cdots, I_{40}$ are algebraically independent. 
Hence the moduli space of cubic surfaces is isomorphic to the weighted 
projective space $\PP(1:2:3:4:5)$. To construct invariants $I_n$ explicitly, the following representation of cubic 
surfaces is useful. A general cubic surface is written as a complete intersection
\[
 S_{\la} \ : \ X_0 + \cdots + X_4 = 0, \quad 
\la_0 X_0^3 + \cdots + \la_4 X_4^3 = 0 
\]
in $\PP^4$ with $\la_0 \cdots \la_4 \ne 0$, which is called the Sylvester form. For a given cubic surface, the parameter 
$\lambda = [\la_0:\cdots:\la_4] \in \PP^4$ is uniquely determined up to permutations. Let $\sigma_i$ be 
the $i$-th elementary symmetric polynomial in $\la_0, \cdots, \la_4$. Then we have
\[
 I_8 = \sigma_4^2 - 4 \sigma_3 \sigma_5, \quad I_{16} = \sigma_5^3 \sigma_1, \quad
I_{24} = \sigma_5^4 \sigma_4, \quad I_{32} =\sigma_5^6 \sigma_2, \quad
I_{40} = \sigma_5^8, \quad I_{100} = \sigma_5^{18} \prod_{0 \leq i < j \leq 4} (\la_i - \la_j)
\]
\begin{rem}
If we have $\sigma_5 = 0$, for example $\la_0 = 0$, then $S_{\la}$ is equivalent to 
a diagonal surface 
\[
 \la_1 X_1^3 + \cdots + \la_5 X_5^3 = 0.
\]
\end{rem}
\subsection{The singular locus}
Let $\Delta_{Sing}(\la)$ be a polynomial
\[
 \Delta_{Sing} (\la) = (\la_0 \la_1 \la_2 \la_3 \la_4)^8 
\prod (\frac{1}{\sqrt{\la_0}} + \frac{\varepsilon_1}{\sqrt{\la_1}} + \frac{\varepsilon_2}{\sqrt{\la_2}} 
+ \frac{\varepsilon_3}{\sqrt{\la_3}} + \frac{\varepsilon_4}{\sqrt{\la_4}})
\]
of degree $32$, where $(\varepsilon_1, \cdots, \varepsilon_4 )$ runs over $\{ \pm1 \}^4$. 
Then $S_{\lambda}$ is smooth iff $\Delta_{Sing}(\la) \ne 0$. Hence we may regard an open set 
$\Lambda = \{ \la \in \PP^4 \ | \ \Delta_{Sing}(\la) \ne 0,\ \sigma_5 \ne 0 \}$ as a parameter space of smooth cubic 
surfaces having Sylvester form. 
\begin{rem}
In terms of classical invarinats, we have
\[
 \Delta_{sing}(\la) = (I_8^2 - 2^6 I_{16})^2 - 2^{14} (I_{32} + 2^{-3}I_8 I_{24}).
\]
\rm{(}As pointed out in \cite{DvG}, the exponent $-3$ in the formula is omitted in \cite{Sa}, p. 198. \rm{)}
\end{rem}
\subsection{Eckardt points}
We have an important geometric interpretation for skew invariant $I_{100}$.
A smooth $S_{\la} \ (\la \in \Lambda)$ has an Eckardt point 
iff $\la_i = \la_j$ for some $i \ne j$, that is, iff 
\[
 \Delta(\la) = \prod_{1 \leq i<j \leq 5} (\la_i - \la_j) = 0.
\]
(If there exists a plane section of a cubic surface which consists of
three lines intersecting at a single point, then this point is called an Eckardt point.) 
\subsection{Hessian K3 surfaces}
The Hessian of $S_{\la}$ is given by
\[
 H_{\la} \ : \ \ X_0 + \cdots + X_4 = 0, \quad 
\frac{1}{\la_0 X_0} + \cdots + \frac{1}{\la_4 X_4} = 0.
\]
It has singularities at $P_{ijk}^{\la} = \{ X_i = X_j = X_k = 0 \} \cap H_{\la}$, and contains lines 
$L_{ij}^{\la} = \{X_i = X_j = 0\} \cap H_{\la}$. 
If $S_{\la}$ is smooth, then ten nodes $P_{ijk}^{\la}$ are all of singularities of 
$H_{\la}$ and we obtain K3 surface $\tilde{H_{\la}}$ by resolving them. 
The transcendental lattice of a general $\tilde{H}_{\la}$ was given in \cite{DoKe}. 
\begin{tm}[Dolgachev and Keum]
Let $\tilde{L}_{ij}^{\la}$ be the proper transform of $L_{ij}$, 
and $E_{ijk}^{\la}$ be the exceptional curve blown-down to $P_{ijk}^{\la}$. 
Let $N_{\la} \subset \mathrm{H}^2(\tilde{H_{\la}}, \ZZ)$ be a sublattice generated by $[\tilde{L}_{ij}^{\la}]$ 
and $[E_{ijk}^{\la}]$.
For a general $\la \in \Lambda$, the Neron-Severi group $\mathrm{NS}(\tilde{H_{\la}})$ coincides 
with $N_{\la}$. In this case, the transcendental lattice 
$M_{\la} = \mathrm{NS}(\tilde{H_{\la}})^{\perp} \subset \mathrm{H}^2(\tilde{H_{\la}}, \ZZ)$ is 
isomorphic to $\mathrm{U} \oplus \mathrm{U}(2) \oplus \mathrm{A}_2(2)$ $(\ZZ^6$ with a bilinear form
$Q = \begin{bmatrix} 0 & 1 \\ 1 & 0 \end{bmatrix} \oplus \begin{bmatrix} 0 & 2 \\ 2 & 0 \end{bmatrix} \oplus 
\begin{bmatrix} -4 & 2 \\ 2 & -4 \end{bmatrix})$.
\end{tm}
In \cite{DvG}, transcendental lattices for several interesting subfamily were studied.
\begin{tm}[Dardanelli and van Geemen] \label{th2}
The transcendental lattice of a Hessian K3 surface of \\
{\rm (1)} a general cubic surface with a node is 
$T_{node} = \left< 2 \right> \oplus \left< 6 \right> \oplus \left< -2 \right>^3$, \\
{\rm (2)} a general smooth cubic surface with an Eckardt point is 
$T_{Eck} = \mathrm{U} \oplus \mathrm{U}(2) \oplus \left< -12 \right>$, \\
{\rm (3)} a general cubic surface which does not admit Sylvester form
is $T_{NS} = \mathrm{U} \oplus \mathrm{U}(2) \oplus \left< -4 \right>$
{\rm (}Such a surface is defined by 
$X_1^3 + X_2^3 + X_3^3 - X_0^2 (\la_0 X_0 + 3 \la_1 X_1 + 3 \la_2 X_2 + 3 \la_3 X_3) = 0$ {\rm )}.
\end{tm}
In \cite{Ro}, the Kummer locus was determined.
\begin{tm}[Rosenberg]
The Hessian $H_{\la} \ (\la \in \Lambda)$ 
is the blow-up of a Weber hexad on a Kummer surface iff
\[
 \Delta_{Km}(\mu) = \sum_{i = 0}^4 \mu_i^3 - \sum_{i \ne j} \mu_i^2 \mu_j + 
2 \sum_{i \ne j \ne k \ne i} \mu_i \mu_j \mu_k = 0
\]
where $\mu_i = \la_i^{-1}$. In classical invariants, this locus is given by $I_8 I_{24} + 8 I_{32} = 0$. 
\end{tm}
\begin{rem}
We have a birational transformation
\[
\iota : H_{\la} \longrightarrow H_{\la}, \quad
[X_0: \cdots : X_4] \mapsto [\frac{1}{\la_0 X_0}: \cdots : \frac{1}{\la_4 X_4}]
\]
which interchanges $P_{ijk}$ with $L_{mn}$ where $\{ i,j,k,m,n \} = \{ 1,2,3,4,5\}$. It
acts on $\tilde{H_{\la}}$ as a fixed-point-free involution. Hence the quotient surface 
$\tilde{H_{\la}}/ \left< \iota \right>$ is an Enriques surface. 
\end{rem}
\section{The period domain and the period mapping} 
\subsection{The period domain}
Let $L$ be the K3 lattice $\mathrm{U}^3 \oplus (\mathrm{E}_8)^2 \cong \mathrm{H}^2(\tilde{H_{\la}}, \ZZ)$. 
We fix a sublattice $N \subset L$ generated by $\ell_{ij}$ and $e_{ijk}$ corresponding to $[\tilde{L}_{ij}^{\la}]$ 
and $[E_{ijk}^{\la}]$, and put $M = N^{\perp} = \mathrm{U} \oplus \mathrm{U}(2) \oplus \mathrm{A}_2(2)$. 
If $\phi_{\la} : \mathrm{H}^2(\tilde{H_{\la}}, \ZZ) \rightarrow L$ is an isomorphism such that 
$\phi([\tilde{L}_{ij}^{\la}]) = \ell_{ij}$ and $\phi([E_{ijk}^{\la}]) = e_{ijk}$, we have $\phi_{\la}(M_{\la}) = M$. 
Since $\mathrm{H}^{2,0}(\tilde{H_{\la}}) \subset M_{\la} \otimes \CC$, the $\CC$-linear extension of $\phi_{\la}$
maps $\Omega \in \mathrm{H}^{2,0}(\tilde{H_{\la}})$ into $M \otimes \CC$. By the  Hodge-Riemann relation 
$\Omega \wedge \Omega =0$ and $\Omega \wedge \bar{\Omega} > 0$, 
the period $\phi_{\la} (\Omega)$ belongs to a domain 
\begin{align*}
\DD_M = \{z \in M \otimes \CC \ |
\  {}^tzQz = 0, \ {}^tzQ\bar{z} > 0 \} / \CC^{\times} \subset \PP(M \otimes \CC).
\end{align*}
More explicitly, we have
\begin{align*}
z = [1: z_2: \cdots : z_6] \in \DD_M \Leftrightarrow \begin{cases} 
z_2 = -2 (z_3 z_4 - z_5^2 + z_5 z_6 - z_6^2) \\
y_3 y_4 - y_5^2 + y_5 y_6 - y_6^2 >0 \ \ (y_i = \mathrm{Im} z_i) \end{cases}
\end{align*}
and $\DD_M = \DD_M^+ \coprod \DD_M^-$ where $\DD_M^{\pm} = \{z \in \DD_M \ : \ \pm y_3 > 0 \}$.
\subsection{Integral orthogonal group}
Let us define orthogonal groups 
\begin{align*}
\mathrm{O}(M) = \{ g \in \mathrm{GL}_6(\ZZ) \ | \ {}^tgQg = Q \}, \qquad
\OM = \{ g \in \mathrm{O}(M) \ | \ g(\DD_M^+) = \DD_M^+ \}.
\end{align*}
They act on $\DD_M$ and $\DD_M^+$ respectively. 
Because $\DD_M^+$ is connected, we see that $g \in \mathrm{O}(M)$ belongs to $\OM$ 
if and only if  $g(p_0) \in \DD_M^+$ for a point $p_0 = [\sqrt{2}:\sqrt{2}:\sqrt{-1}:\sqrt{-1}:0:0] \in \DD_M^+$.  
\\
\indent
We define the discriminant form
\[
 q_M : \check{M}/M \longrightarrow \QQ / 2 \ZZ, \quad  x \mapsto {}^txQx,
\]
and the orthogonal group
\[
 \mathrm{O}(q_M) = \{ g \in \mathrm{Aut}(\check{M}/M) \ | \ q_M(gx) = q_M(x) \}.
\]
Let $\OMK$ be the kernel of a natural homomorphism $\OM \rightarrow \mathrm{O}(q_M)$. 
We can lift any $g \in \OMK$ to an automorphism of $L$ acting on $N$ trivially (\cite{Ni}). We consider 
also special orthogonal groups
\[
 \SOM = \OM \cap \SL_6(\ZZ), \quad \SOMK = \OMK \cap \SL_6(\ZZ).
\]
We will show the following Proposition later.
\begin{prop} \label{quotient group}
We have 
\[
 \OM / \OMK \cong \SOM / \SOMK \cong \mathrm{O}(q_T) \cong \{ \pm 1 \} \times \mathrm{S}_5, 
\]
where $\{ \pm 1 \}$ acts on $\DD_M^+ / \OMK$ trivially. 
\end{prop}
\begin{rem}
In \cite{DoKe}, it was shown that automorphisms of a certain domain of the positive cone in $N$ is 
$\{ \pm 1 \} \times \mathrm{S}_5$, which can be realized as the subgroup of
$\mathrm{Aut}(N)$ generated by the Enriques involution and the group of symmetries of the
Sylvester pentahedron.
\end{rem}
Let $\OME$ be the kernel of the projection $\OM \rightarrow \OM / \OMK \rightarrow \mathrm{S}_5$, and put 
\[
 \SOME = \OME \cap \SL_6(\ZZ).
\]
Then $\OMK$ (resp. $\SOMK$) is a subgroup of $\OME$ (resp. $\SOME$) of index $2$.
\subsection{Marked Hessian K3 surfaces and the period mapping}
We call $(\tilde{H_{\la}}, \{[\tilde{L}_{ij}^{\la}]\}, \{[E_{ijk}^{\la}]\}, \phi_{\la})$
a Hessian K3 surface with Slvester structure, if $\phi_{\la} : \mathrm{H}^2(\tilde{H_{\la}}, \ZZ) \rightarrow L$ 
is an isomorphism such that $\phi([\tilde{L}_{ij}^{\la}]) = \ell_{ij}$, $\phi([E_{ijk}^{\la}]) = e_{ijk}$ and 
$\phi_{\la}(\Omega_{\la}) \in \DD_M^+$. 
From the Torelli theorem (\cite{PS}) and the above Proposition, we see that
\begin{tm}
The period map
\[
 \{ (\tilde{H_{\la}}, \{[\tilde{L}_{ij}^{\la}]\}, \{[E_{ijk}^{\la}]\}, \phi_{\la}) \} \longrightarrow \DD_M^+, \quad
(\tilde{H_{\la}}, \{[\tilde{L}_{ij}^{\la}]\}, \{[E_{ijk}^{\la}]\}, \phi_{\la}) \mapsto \phi_{\la}(\Omega_{\la})
\]
induces $S_5$-equivariant injective map $\Lambda \rightarrow \DD_M^+ / \OMK \cong \DD_M^+ / \OME$.
\end{tm}
\begin{rem}
The period domain $\DD_{Enr}$ of Enriques surfaces is the domain of type IV defined by 
the lattice $L_{Enr} = \mathrm{U} \oplus \mathrm{U}(2) \oplus \mathrm{E}_8(2)$. We have a primitive embedding 
$M \subset L_{Enr}$ and  $\DD_M^+ \subset \DD_{Enr}$. Automorphic forms on $\DD_{Enr}$ are given in \cite{K2} and 
\cite{FS}.
\end{rem}
\section{Two Arithmetic quotients}
We give an explicit isomorphism between the period domain $\DD_M^+$ and the Hermitian upper half 
space of degree $2$, and compare action of two discrete groups.
\subsection{Hermitian upper half space}
The Hermitian upper half space of degree $2$ is defined by
\[
 \HH_2 = \{ \tau \in \GL_2(\CC) \ | \ \frac{1}{2i}(\tau - \tau^*) > 0  \}.
\]
We have an isomorphism
\[
\Psi :  \DD_M^+ \longrightarrow \HH_2, \qquad [1:z_2:\cdots:z_6] \mapsto
\begin{bmatrix} z_3 & z_5 + \w z_6 \\ z_5 + \w^2 z_6 & z_4 \end{bmatrix}
\]
where $\w = e^{2 \pi i /3}$. Note that $z_2 = -2 \det \Psi(z)$. 
The modular group 
\[
 \HG = \{ g \in \GL_4(\ZZ[\w]) \ | \ g^* J g = J \}, \quad J = \begin{bmatrix} 0 & \I_2 \\ -\I_2 & 0 \end{bmatrix}
\]
acts on $\HH_2$ by $\begin{bmatrix} A & B \\ C & D \end{bmatrix} \cdot \tau = (A \tau +B)(C \tau +D)^{-1}$, and
we have an involution $T : \HH_2 \rightarrow \HH_2, \ \tau \mapsto {}^t\tau$.
\begin{rem}
For $g \in \HG$, we have $T \cdot g \cdot \tau = \bar{g} \cdot T \cdot \tau$.
\end{rem}
We consider the following congruence subgroups
\begin{align*}
\HG_0(2) = \{ \begin{bmatrix} A & B \\ C & D \end{bmatrix} \in \HG \ | \ C \equiv 0 \mod 2 \}, \quad
\HG_1(2) = \{ \begin{bmatrix} A & B \\ C & D \end{bmatrix} \in \HG_0(2) \ | \ A \equiv \I_2 \mod 2 \}. 
\end{align*}
To study generators of $\HG_0(2)$ and $\HG_1(2)$, let us define a group
\[
 G(2) = \{ g \in \GL_2(\ZZ[\w]) \ | \ g \equiv \I_2 \mod 2\}.
\]
\begin{lem} \label{division-lemma} If $\begin{bmatrix} \alpha \\ \beta \end{bmatrix} \in \ZZ[\w]^2$ satisfys 
$\begin{bmatrix} \alpha \\ \beta \end{bmatrix} \equiv \begin{bmatrix} 1 \\ 0 \end{bmatrix} \mod 2$, then 
there exists $A \in G(2)$ such that 
\[
 A \begin{bmatrix} \alpha \\ \beta \end{bmatrix} = \begin{bmatrix} \delta \\ 0 \end{bmatrix} 
\quad (\delta \in \ZZ[\w]).
\]
\end{lem}
\begin{proof}
Let $\delta$ be a generator of an ideal $(\alpha, \beta)$ of $\ZZ[\w]$, and put 
$\alpha' = \alpha / \delta$ and $\beta' = \beta / \delta$. We may assume that $\alpha' \equiv 1 \mod 2$ by replacing 
$\delta$ with $\w \delta$ or $\w^2 \delta$ if necessary. We see that 
$A = \begin{bmatrix} x - \beta' y & y + \alpha' y \\ -\beta' & \alpha' \end{bmatrix}$ is 
a desired matrix for $x, y \in \ZZ[\w]$ such that $\alpha' x + \beta' y = 1$.
\end{proof}
\begin{prop} \label{unitary-prop}
{\rm (1)} The group $\HG_1(2)$ is generated by 
\[
 g(A) = \begin{bmatrix} A & 0 \\ 0 & {}^t \bar{A}^{-1} \end{bmatrix}, \quad
g(B)^* = \begin{bmatrix} \I_2 & B \\ 0 & \I_2 \end{bmatrix}, \quad
g(B)_* = \begin{bmatrix} \I_2 & 0 \\ 2B & \I_2 \end{bmatrix}
\] 
where $A \in G(2)$ and $B = \begin{bmatrix} m_1 & m_3 + \w m_4 \\ m_3 + \w^2 m_4 & m_2 \end{bmatrix}$ 
$(m_1, \cdots, m_4 \in \ZZ)$. 
\\
{\rm (2)} We have an exact sequence
\[
1 \longrightarrow \HG_1(2) \longrightarrow \HG_0(2) \overset{f}{\longrightarrow} 
\GL_2(\FF_4) \longrightarrow 1, \quad f(\begin{bmatrix} A & B \\ C & D \end{bmatrix}) = A \mod 2
\]
and the group $\HG_0(2)$ is generated by $\HG_1(2)$ and $g(A)$ with $A \in \GL_2(\ZZ[\w])$.
\end{prop}
\begin{proof}
(1) Let $\Gamma$ be the group generatd by matrices in the Proposition, and 
$g = \begin{bmatrix} A & B \\ C & D \end{bmatrix}$ be in $\HG_1(2)$. If $C = 0$, then we see easily that $g \in \Gamma$. 
So we show that there exists $g' \in \Gamma$ such that $g' g$ is a matrix with $C = 0$. This is proved by 
the following division algorithm.
\\ \indent
(i) By Lemma \ref{division-lemma}, there exist $A \in G(2)$ such that 
the first column of $g(A) g$ is $x = {}^t(\alpha, 0, \gamma, \delta)$. Then we have $\alpha \in \QQ \gamma$
by the unitary condition ${}^t x J \bar{x} = 0$. Therefore we may assume that the first column of 
$g$ is ${}^t(m \alpha,\ 0,\ 2n \alpha,\ 2\beta)$ with $\alpha, \beta \in \ZZ[\w]$, $m, n \in \ZZ$ and 
$m \equiv 1 \mod 2$. Multiplying $g(B)^*$ and $g(B)_*$ with $B = \begin{bmatrix} \pm 1 & 0 \\ 0 & 0 \end{bmatrix}$, 
we have transformations of the first column
\begin{align*}
g(B)^* : (m \alpha,\ 0,\ 2n \alpha,\ 2\beta) &\mapsto ((m \pm 2n) \alpha,\ 0,\ 2n \alpha,\ 2\beta), \\
g(B)_* :(m \alpha,\ 0,\ 2n \alpha,\ 2\beta) &\mapsto (m \alpha,\ 0,\ 2(n \pm m) \alpha,\ 2\beta).
\end{align*}
By these transformaions, we can change the value of $n$ into $0$. 
\\ \indent
(ii) Let $(\alpha,0,0,2\beta)$ be the first column of $g$.
\\
(ii-1) If $|\alpha| < \sqrt{3} |\beta|$, then let us consider a transformation 
\[
 g(B)_* : 
(\alpha,0,0,2\beta) \mapsto (\alpha, 0, 0, 2(\beta -\e \alpha)), \qquad
B = \begin{bmatrix} 0 & -\bar{\e} \\ -\e & 0 \end{bmatrix}
\]
for $\e \in \ZZ[\w]^{\times}$ such that  
$-\frac{\pi}{6} \leq \mathrm{arg}(\bar{\beta} \e \alpha) \leq \frac{\pi}{6}$. Then we have
\begin{align*}
|\beta - \e \alpha|^2 = |\beta|^2 - 2 \RE (\bar{\beta} \e \alpha) + |\alpha|^2 
&\leq |\beta|^2 -2 \cos \frac{\pi}{6} |\beta| |\alpha| + |\alpha|^2 \\
&= |\beta|^2 - (\sqrt{3} |\beta| - |\alpha|)|\alpha|  < |\beta|^2.
\end{align*}
(ii-2) If $\sqrt{3} |\beta| \leq |\alpha|$, then let us consider a transformation
\[
 g(B)^* : 
(\alpha, 0, 0, 2\beta) \mapsto (\alpha - \e \beta, 0, 0, 2\beta), \qquad
B = \begin{bmatrix} 0 & -\e \\ -\bar{\e} & 0 \end{bmatrix}
\]
for $\e \in \ZZ[\w]^{\times}$ such that 
$-\frac{\pi}{6} \leq \mathrm{arg}(\bar{\alpha} \e \beta) \leq \frac{\pi}{6}$. Then we have
\begin{align*}
|\alpha - 2 \e \beta|^2 = |\alpha|^2 - 4 \RE (\bar{\alpha} \e \beta) + 4 |\beta|^2 
&< |\alpha|^2 -4 \cos \frac{\pi}{6} |\alpha| |\beta| + 6 |\beta|^2 \\
&= |\alpha|^2 - 2\sqrt{3}(|\alpha| - \sqrt{3}|\beta|)|\beta| < |\alpha|^2.
\end{align*} 
Repeating the above transformations, we may assume that $\beta = 0$.
\\ \indent
(iii) Let $x_1 = {}^t(\e, 0, 0, 0)$ be the first column of $g$, $x_2 = {}^t(\alpha, \beta, \gamma, \delta)$ 
be the second column. Then we have $\gamma = 0$ and $\beta \in \QQ \delta$ since ${}^tx_1J \bar{x_2} = 0$ and 
${}^tx_2J \bar{x_2} = 0$. Applying the same argument with (i), we can change $\delta$ into $0$.
\\
(2) is easily shown.
\end{proof}
\begin{rem}
We have an isomorphism $\GL_2(\FF_4) / \FF_4^{\times} \cong \mathrm{A}_5$ as 
even permutations of five points of $\PP^1(\FF_4)$.
\end{rem}
\subsection{Comparison of two groups}
We compare $\OM$ and $\HG$ as automorphisms of $\DD_M^+ \cong \HH_2$. 
\begin{lem} \label{translations}
Let $h \in \GL_6(\ZZ)$ be a matrix of the following form
\begin{align*}
h = \begin{bmatrix} a_{11} & 0 & 0 & 0 & 0 & 0\\
a_{21} & a_{22} & a_{23} & a_{24} & a_{25} & a_{26}\\ 
a_{31} & 0 & 1 & 0 & 0 & 0\\
a_{41} & 0 & 0 & 1 & 0 & 0\\
a_{51} & 0 & 0 & 0 & 1 & 0\\
a_{61} & 0 & 0 & 0 & 0 & 1\end{bmatrix}.
\end{align*}
We have $h \in \OM$ if and only if the conditions
\begin{itemize}
\item [{\rm (i)}] $a_{11} = a_{22} = 1$
\item [{\rm (ii)}] $a_{21} = -\frac{1}{2}{}^ta Q' a, \quad 
a = {}^t(a_{31},a_{41},a_{51},a_{61}), \quad Q' = \begin{bmatrix} 0 & 2 \\ 2 & 0 \end{bmatrix}
\oplus \begin{bmatrix} -4 & 2 \\ 2 & -4 \end{bmatrix}$
\item [{\rm (iii)}] $(a_{23}, a_{24},a_{25},a_{26}) = - {}^ta Q'$
\end{itemize}
are satisfied. Therefore any $(a_{31},a_{41},a_{51},a_{61}) \in \ZZ^4$ 
determines $h = h(a_{31}, \cdots, a_{61}) \in \OM$ 
by the relations {\rm (i)} - {\rm (iii)}. Moreover, we have 
\begin{align*}
h(a_{31}, \cdots, a_{61})h(b_{31}, \cdots ,b_{61}) = 
h(a_{31} + b_{31},\cdots,a_{61} + b_{61}).
\end{align*}
Hence they form a subgroup of $\OM$, which is isomorphic to $\ZZ^4$. As an automorphism of $\HH_2$, we have
\[
  \Psi(h(m_1,m_2,m_3,m_4) \cdot z) = g(B)^* \cdot \Psi(z), \quad
B = \begin{bmatrix} m_1 & m_3 + \w m_4 \\ m_3 + \w^2 m_4 & m_2 \end{bmatrix}.
\]
\end{lem}
\begin{proof}
Since $\det h = \pm1$ and $h(p_0) \in \DD_M^+$, 
we see that $a_{11} = a_{22} =1$. Other conditions are obtained from the condition ${}^thQh = Q$.\\
\end{proof}
\begin{lem} \label{uni}
{\rm (1)} A map $\psi : \GL_2(\ZZ[\w]) \rightarrow \SOM$,
\[ 
\begin{bmatrix} a_1 & a_2 \\ a_3 & a_4 \end{bmatrix} \mapsto 
\I_2 \oplus \frac{2}{\sqrt{3}} \begin{bmatrix} 
\sqrt{3} |a_1|^2 /2 & \sqrt{3} |a_2|^2 /2 & \sqrt{3} \RE(a_1 \overline{a_2}) & \sqrt{3} \RE(\w a_1 \overline{a_2}) \\
\sqrt{3} |a_3|^2 /2 & \sqrt{3} |a_4|^2 /2 & \sqrt{3} \RE(a_3 \overline{a_4}) & \sqrt{3} \RE(\w a_3 \overline{a_4}) \\ 
\IM (\w a_3 \overline{a_1})  & \IM (\w a_4 \overline{a_2}) & \IM (\w (a_4 \overline{a_1} + a_3 \overline{a_2}) ) & 
\IM (a_4 \overline{a_1} - \w a_2 \overline{a_3} ) \\ 
\IM (a_1 \overline{a_3}) & \IM (a_2 \overline{a_4}) & \IM (a_1 \overline{a_4} + a_2 \overline{a_3}) &
\IM (\w (a_1 \overline{a_4} - a_3 \overline{a_2})) 
\end{bmatrix} 
\]
is a homomorphism such that $\mathrm{Ker} \ \psi = \{ \pm1,\ \pm \w,\ \pm \w^2\}$ and 
$\Psi(\psi(A) \cdot z) = g(A) \cdot \Psi(z)$. Moreover we have $\psi(g) \equiv \I_6 \mod 2$ if and only if $g \in G(2)$.
\\
{\rm (2)} For $u_1 = \I_4 \oplus \begin{bmatrix} 1 & -1 \\ 0 & -1 \end{bmatrix} \in \OM$, we have
$\Psi(u_1 \cdot z) = {}^t \Psi(z) = T \cdot \Psi(z)$.
\end{lem}
\begin{proof}
The proof is straight-foward.
\end{proof}
\begin{lem}
A subset
\[
 \OM_0 = \{ [a_{ij}] \in \OM : \begin{bmatrix}a_{11} & a_{12} \\ a_{21} & a_{22} \end{bmatrix} \equiv
 \begin{bmatrix}1 & 0 \\ 0 & 1 \end{bmatrix} \mod 2\}
\]
is a normal subgroup of $\OM$. and we have 
\[
 \OM = \OM_0 \cup g_0 \OM_0, \qquad g_0 = \left[ \begin{array}{cc} 0 & 1 \\ 1 & 0 \\ \end{array} \right] \oplus \I_4.
\]
\end{lem}
\begin{proof}
For $g = [\mathbf{a}_1, \cdots, \mathbf{a}_6] = [a_{ij}] \in \OM$, we have
\begin{align*}
0 &= \frac{1}{2} {}^t\mathbf{a}_k Q \mathbf{a}_k = a_{1k}a_{2k} + 
2(a_{3k} a_{4k} - a_{5k}^2 + a_{5k} a_{6k} - a_{6k}^2) \qquad (k=1, 2) \\
1 &= {}^t\mathbf{a}_1 Q \mathbf{a}_2 = a_{11} a_{22} + a_{12} a_{21} + 2(a_{13}a_{24} + a_{14}a_{23} 
- 2a_{15}a_{25} + a_{15}a_{26} +a_{16}a_{25} -2 a_{16} a_{26}).
\end{align*}
Therefore we see that
\[
 \begin{bmatrix}a_{11} & a_{12} \\ a_{21} & a_{22} \end{bmatrix} \equiv
 \begin{bmatrix}1 & 0 \\ 0 & 1 \end{bmatrix} \ \text{or} \ 
\begin{bmatrix}0 & 1 \\ 1 & 0 \end{bmatrix} \mod 2.
\]
The conditions ${}^t\mathbf{a}_1 Q \mathbf{a}_k = {}^t\mathbf{a}_2 Q \mathbf{a}_k = 0 \ (k=3,4,5,6)$ imply 
$a_{1k} \equiv a_{2k} \mod 2 \ (k = 3,4,5,6)$. 
Using this equalities, we can check that $gh \in \OM_0$ for $g, h \in \OM_0$. 
\end{proof}
Let us consider the following elements of $\OM$.
\[
 h_1 = h(1,0,0,0), \quad h_2 = h(0,1,0,0), \quad h_3 = h(0,0,1,0), \quad h_4 = h(0,0,0,1), \quad h_i' = g_0 h_i g_0,
\]
\[
g_1 = \I_2 \oplus \left[ \begin{array}{cccc}
1 & 0 & 0 & 0 \\ 
1 & 1 & 2 & -1 \\ 
1 & 0 & 1 & 0 \\ 
0 & 0 & 0 & 1 \end{array} \right], \quad
g_2 = \I_2 \oplus \left[ \begin{array}{cccc}
1 & 0 & 0 & 0 \\ 
1 & 1 & -1 & 2 \\ 
0 & 0 & 1 & 0 \\ 
1 & 0 & 0 & 1 \end{array} \right]
\]
\[
\I_{4,2} = \I_4 \oplus (-\I_2), \quad u_0 = \I_2 \oplus 
\left[ \begin{array}{cc} 0 & 1 \\ 1 & 0 \\ \end{array} \right] \oplus \I_2, \quad 
u_1 = \I_4 \oplus \left[ \begin{array}{cc} 1 & -1 \\ 0 & -1 \\ \end{array} \right], \quad
u_2 = \I_4 \oplus \left[ \begin{array}{cc} 0 & -1 \\ 1 & -1 \\ \end{array} \right]
\] 
Put $\SOM_0 = \OM_0 \cap \SOM$. We have 
\[
 \SOM = \SOM_0 \cup (g_0 u_0 \I_{4,2})\SOM_0, \quad \OM = \OM_0 \cup (g_0 u_0 \I_{4,2})\OM_0.
\]
The action of $g_0u_0 \I_{4,2}$ on $\HH_2$ is given by 
\[
\Phi(g_0 u_0 \I_{4,2} \cdot z) = -\frac{1}{2}\Psi(z)^{-1} = W \cdot \Psi(z), \qquad
W = \begin{bmatrix} 0 & -\I_2 \\ 2 \I_2 & 0 \end{bmatrix}.
\]
The involution $W$ ia a normalizer of $\HG_0(2)$ and $\HG_1(2)$. We denote semi-direct products 
$\HG_0(2) \rtimes \left< W \right>$ and $\HG_1(2) \rtimes \left< W \right>$ by $\HG_0^*(2)$ and 
$\HG_1^*(2)$ respectively.
\begin{prop} \label{orth-gen}
The group $\SOM_0$ is generated by
\[
  h_i, \ h_i' \ (i = 1,2,3,4), \quad g_1, \ g_2, \ u_0 g_1 u_0, \quad \pm \I_{2,4}, \quad u_0 u_1, \ u_2.
\]
\end{prop} 
\begin{proof} 
We can prove by the same method with the proof of Proposition \ref{unitary-prop}. 
Let $G$ be the group generated by matrices in the Proposition. We show that 
for any $X \in \SOM_0$, there exists $g \in G$ such that $gX \in G$.
\\ \indent
(i) Let $\mathbf{a}_2 = {}^t(a_1, \cdots, a_6)$ be the second column of $X \in \SOM_0$. 
By transformations
\[
 h_2 : (a_2, a_3) \mapsto (a_2 - 2a_3, a_3), \qquad 
 h_1' : (a_2, a_3) \mapsto (a_2, a_2 + a_3),
\]
we can change the value $a_2 a_3$ into $0$. Then we hvae $a_3 =0$ since $a_2 \equiv 1 \mod 2$. 
By the condition ${}^t\mathbf{a}_2Q\mathbf{a}_2 = 0$, we have
\[
 a_1 a_2 - 2(a_5^2 - a_5 a_6 + a_6^2) = 0 \qquad \therefore \
a_1 a_2 = 2|a_5 + \w a_6|^2.
\]
Put $z(\mathbf{a}_2) = a_5 + \w a_6$. Because we have
\[
 z(\I_{4,2} \cdot \mathbf{a}_2) = - z(\mathbf{a}_2), \qquad z(u_2 \cdot \mathbf{a}_2) = \w z(\mathbf{a}_2),
\]
we may assume that $\theta = \mathrm{arg} (z(\mathbf{a}_2))$ satisfys 
$-\frac{\pi}{6} \leq \theta \leq \frac{\pi}{6}$.
If $z(\mathbf{a}_2) \ne 0$, then we have $|a_1| < \sqrt{3}|z(\mathbf{a}_2)|$ or  $|a_2| < \sqrt{3}|z(\mathbf{a}_2)|$.
If $|a_i| < \sqrt{3}|z(\mathbf{a}_2)|$, then we have
\[
 |z(\mathbf{a}_2) - a_i|^2 = |z(\mathbf{a}_2)|^2 - 2 \cos \theta |z(\mathbf{a}_2)||a_i| + |a_i|^2 
\leq |z(\mathbf{a}_2)|^2 -(\sqrt{3}|z(\mathbf{a}_2)| -|a_i|)|a_i| < |z(\mathbf{a}_2)|.
\] 
Repeating transformations
\[
 z(h_3^{-1} \cdot \mathbf{a}_2) = z(\mathbf{a}_2) - a_1, \qquad
z((h_3')^{-1} \cdot \mathbf{a}_2) = z(\mathbf{a}_2) - a_2,
\]
we can change $z(\mathbf{a}_2)$ into $0$, that is, $\mathbf{a}_2$ into ${}^t(0,a_2,0,a_4,0,0)$.
Now multiplying $h_1, \ h_2', \ -\I_{4,2}$, we may assume that $\mathbf{a}_2 = {}^t(0,1,0,0,0,0)$. By the 
condition $^t{}XQX = Q$, we see that $X$ is the following form
\[
 X = \left[
\begin{array}{cc|cccc} 
1 & 0 & 0 & 0 & 0  & 0 \\
* & 1 & * & * & * & * \\ 
\hline
* & 0 &  &  &  &  \\
* & 0 &  & X' &  &  \\ 
* & 0 &  &   &  &  \\
* & 0 &  &  &  & 
\end{array} \right], \quad \I_2 \oplus X' \in \SOM_0.
\] 
Because $(\I_2 \oplus X')^{-1}X$ is a matrix of the form in Lemma \ref{translations}, we may assume that 
$X = \I_2 \oplus X'$.
\\ \indent
(ii) Let $\mathbf{a}_4 = {}^t(0,0,b_3,b_4,b_5,b_6)$ be the $4$-th column of $X = \I_2 \oplus X' \in \SOM_0$. 
We have
\[
 b_3 b_4 = |b_5 + \w b_6|^2
\]
since ${}^t\mathbf{a}_4Q\mathbf{a}_4 = 0$. By the similar arguments with (i), we can change $b_5 + \w b_6$ into $0$ 
using $g_1$ and $g_0g_1g_0$. Multiplying $u_0 u_1$ if necessary, we have $\mathbf{a}_4 = {}^t(0,0,0,1,0,0)$.
By the condition ${}^t X Q X = Q$,  we see that
\[
 X = \I_2 \oplus \left[ \begin{array}{cccc}
1 & 0 & 0 & 0 \\
* & 1 & * & * \\
* & 0 & a & b \\ 
* & 0 & c & d  \end{array} \right], \qquad
{}^t \begin{bmatrix} a & b \\ c & d  \end{bmatrix} 
\begin{bmatrix} -2 & 1 \\ 1 & -2  \end{bmatrix}
\begin{bmatrix} a & b \\ c & d  \end{bmatrix} = 
\begin{bmatrix} -2 & 1 \\ 1 & -2  \end{bmatrix}, \qquad ad-bc=1.
\]
Now it is easy to check that $X$ is obtaind from $\I_{4,2}$, $u_2$ and 
\[
 X = \I_2 \oplus \left[ \begin{array}{cccc}
1 & 0 & 0 & 0 \\
4(m^2 -mn + n^2) & 1 & 4m-2n & -2m+4n \\
2m & 0 & 1 & 0 \\ 
2n & 0 & 0 & 1  \end{array} \right]
= g_1^{2m} g_2^{2n}. 
\]  
\end{proof}
\begin{tm}  \label{group-iso} We have isomorphisms 
\begin{align*}
 \SOM_0 / \{\pm 1\} \cong \HG_0(2) / \{\pm1, \pm \w, \pm \w^2\}, \qquad 
\OM_0 / \{\pm 1\} \cong \HG_0(2) / \{\pm1, \pm \w, \pm \w^2\} \rtimes \left< T \right>, \\
\SOM / \{\pm 1\} \cong \HG_0^*(2) / \{\pm1, \pm \w, \pm \w^2\}, \qquad 
\OM / \{\pm 1\} \cong \HG_0^*(2) / \{\pm1, \pm \w, \pm \w^2\} \rtimes \left< T \right>
\end{align*}
as automorphisms of $\DD_M^+ \cong \HH_2$.
\end{tm}
\begin{proof}
For generators in Proposition \ref{orth-gen}, we have
\begin{align*}
 g_1 = \psi(\begin{bmatrix} 1 & 0 \\ 1 & 1 \end{bmatrix}), \quad 
g_2 = \psi(\begin{bmatrix} 1 & 0 \\ \w^2 & 1 \end{bmatrix}), \quad 
u_0 g_1 u_0 = \psi(\begin{bmatrix} 1 & 1 \\ 0 & 1 \end{bmatrix}),\\
\I_{2,4} = \psi(\begin{bmatrix} 1 & 0 \\ 0 & -1 \end{bmatrix}), \quad
u_0 u_1 = \psi(\begin{bmatrix} 0 & 1 \\ 1 & 0 \end{bmatrix}), \quad
u_2 = \psi(\begin{bmatrix} 0 & -1 \\ 1 & -1 \end{bmatrix}),
\end{align*}
and $ h_i, \ h_i'$ correspond to $\begin{bmatrix} \I_2 & B \\ 0 & \I_2 \end{bmatrix}$,
$\begin{bmatrix} \I_2 & 0 \\ 2B & \I_2 \end{bmatrix}$. Considering these correspondence together with 
Proposition \ref{unitary-prop}, Lemma \ref{translations} and Lemma \ref{uni}, we see that 
\[
 \SOM_0 / \{\pm 1\} \cong \HG_0(2) / \{\pm1, \pm \w, \pm \w^2\}.
\]
Other isomorphisms hold since $g_0u_0 \I_{4,2}$ and $u_1$ correspond to $W$ and $T$ respectively, and since we have
\[
 \SOM = \SOM_0 \cup (g_0u_0 \I_{4,2})\SOM_0, \qquad \OM_0 = \SOM_0 \cup u_1 \SOM_0.
\]
\end{proof}
\subsection{The orthogonal group with level structure} 
Now we show Proposition \ref{quotient group}. 
Let $e_1, \cdots, e_6 $ be the standard basis of $M = \ZZ^6$. The dual lattice
$\check{M}$ is generated by
\[
 e_1, \quad e_2, \quad d_1 = \frac{1}{2}e_3, \quad d_2 = \frac{1}{2}e_4, \quad 
d_3 = \frac{1}{6}e_5 + \frac{1}{3}e_6, \quad d_4 = \frac{1}{3}e_5 + \frac{1}{6}e_6,
\]
and $\check{M}/M \cong (\ZZ / 2 \ZZ)^4 \oplus (\ZZ / 3 \ZZ)$ is generated by $d_1, d_2, d_3, d_4$. 
Elements of order $3$ in $\check{M}/M$ are $\pm 2d_3$ ($-2d_3 \equiv 2d_4 \mod \ZZ$). 
Elements of order $2$ in $\check{M}/M$ are
\[
 v_1 = d_1, \quad v_2 = d_2, \quad, v_3 = d_1 + d_2 + d_3 + d_4, \quad 
v_4 =d_1 + d_2 + 3 d_3, \quad v_5 = d_1 + d_2 + 3 d_4
\]
for which we have $q_M(v_i) = 0$, and
\begin{align*}
d_1 + d_2, \quad d_3 + d_4, \quad 3 d_3, \quad 3 d_4, \quad d_1 + d_3 + d_4 \quad
d_1 + 3 d_3, \quad d_1 + 3 d_4 \\
d_2 + d_3 + d_4, \quad d_2 + 3 d_3, \quad d_2 + 3 d_4
\end{align*}
for which the value of $q_M$ is $1$. Hence $\mathrm{O}(q_M)$ acts on $\{v_1, \cdots, v_5\}$ as permutations.
We have the following correspondence 
\[
 g_1 = (14)(35), \quad g_2 = (15)(34), \quad u_0 = (12), \quad
u_1 = (35), \quad u_2 = (345)
\]
Hence the composition $\OM \rightarrow \mathrm{O}(q_M) \rightarrow \mathrm{S}_5$ is surjective. 
Because $v_1, \cdots, v_5$ generate elements of order $2$ in $\check{M}/M$, 
we may identify $\mathrm{Ker}( \mathrm{O}(q_M) \rightarrow S_5 )$ with 
$\mathrm{Aut}(\ZZ / 3 \ZZ) = \mathrm{Aut}\{0,\ 2d_3,\ -2d_3 \}$. 
We see that this group is given by $\{ \mathrm{id}, \ -\I_6 \}$. Hence $\OM \rightarrow \mathrm{O}(q_M)$ is surjective and 
we have $\mathrm{O}(q_M) \cong \mathrm{S}_5 \times \{ \mathrm{id},\ -\I_6 \}$. Because we have 
$\OM = \SOM \cup g_0 \SOM$ and $g_0 \in \OMK$, we see that
\[
 \OM / \OMK \cong \SOM / \SOMK.
\]
\begin{lem} \label{cong-lem}
Let $g = [\mathbf{a}_1, \cdots, \mathbf{a}_6]$ be an element of $\OM$. 
We have $g \in \OMK$ if and only if
\begin{align*}
 &\mathbf{a}_3 \equiv e_3, \quad \mathbf{a}_4 \equiv e_4 \mod 2 \\
 &\mathbf{a}_5 + 2 \mathbf{a}_6 \equiv e_5 + 2 e_6, \quad
 2 \mathbf{a}_5 + \mathbf{a}_6 \equiv 2 e_5 + e_6 \mod 6.
\end{align*} 
\end{lem}
\begin{proof}
This follows from the fact that $g \in \OMK$ iff $g d_i \equiv d_i \mod \ZZ \ (i =1,2,3,4)$.
\end{proof}
\begin{lem} \label{cong-lem2}
Let $g = [\mathbf{a}_1, \cdots, \mathbf{a}_6]$ be an element of $\OM$. 
We have $g \in \OME$ if and only if
\begin{align*}
 \mathbf{a}_i \equiv e_i \mod 2 \qquad (i = 3, \cdots, 6).
\end{align*}
\end{lem}
\begin{proof}
Let us recall that $\OME$ is defined as the kernel of the composition 
\[
 \OM \longrightarrow \mathrm{S}_5 \times \{ \pm 1\} \longrightarrow \mathrm{S}_5.
\]
Therefore $g \in \OM$ belongs to $\OME$ if and only if $g v_i = v_i \ (i=1, \cdots, 5)$, that is, 
$g$ acts trivilaly on $2$-torsions of $\check{M}/M$.
\end{proof}
\begin{tm} 
As automorphisms of $\DD_M^+ \cong \HH_2$, we have isomorphisms
\begin{align*}
\SOME / \{\pm 1\} \cong \HG_1(2) / \{\pm1 \}, \quad
\OME / \{\pm 1\} \cong \HG_1(2) / \{\pm1 \} \rtimes \left< W' \right>
\end{align*}
where $W' = T \cdot [A] \cdot W$ and $A = \begin{bmatrix} 0 & 1 \\ 1 & 0 \end{bmatrix}$.
\end{tm}
\begin{proof}
By Theorem \ref{group-iso}, we have isomorphism 
\[
 f : \HG_0^*(2) / \{\pm1, \pm \w, \pm \w^2 \} \longrightarrow \SOM / \{\pm 1\}.
\]
By Lemma \ref{translations}, Lemma \ref{uni} and Lemma \ref{cong-lem2}, we see that the generator

\[
 \begin{bmatrix} \I_2 & B \\ 0 & \I_2 \end{bmatrix}, \quad
\begin{bmatrix} \I_2 & 0 \\ 2B & \I_2 \end{bmatrix} 
= W \begin{bmatrix} \I_2 & -B \\ 0 & \I_2 \end{bmatrix} W^{-1}, \quad
G(2)
\]
of $\HG_1(2)$ are mapped into $\SOME / \{ \pm 1\}$ by $f$. Considering isomorphisms 
\[
 \SOM / \SOME \cong \mathrm{S}_5, \qquad \HG_0(2) / (\HG_1(2) \times \{1,\w,\w^2\}) 
\cong \GL_2(\FF_4) / \FF_4^{\times} \cong \mathrm{A}_5
\]
we see that $\SOME / \{\pm 1\} \cong \HG_1(2) / \{\pm1 \}$. Since we have
\[
 \OME = \SOME \cup (g_0 \I_{4,2})\SOME
\]
and $g_0 \I_{4,2} \cdot z = W' \cdot \Psi(z)$, we see that
\[
 \OME / \{\pm 1\} \cong \HG_1(2) / \{\pm1 \} \rtimes \left< W' \right>.
\]
\end{proof}
\section{Heegner divisors}
\begin{prop}
Heegener divisors (1), (2) and (3) in Theorem \ref{th2} and the Kummer locus correspond to
$\HG_0^*(2) \rtimes \left< T \right>$-orbit of
\begin{align*}
(1) \ \mathcal{H}_{node} = \{ \tau \in \HH_2 \ | \  2 \det \tau = -1 \}, \quad
(2) \ \mathcal{H}_{Eck} = \{ \tau \in \HH_2 \ | \ {}^t \tau = -\tau \}, \\
(3) \ \mathcal{H}_{NS} = \{ \tau \in \HH_2 \ | \ {}^t \tau = \tau \}, \quad
(4) \ \mathcal{H}_{Km} = \{ \tau + \frac{1}{2} \begin{bmatrix} 0 & \w \\ \w^2 & 0 \end{bmatrix}
 \ | \ \tau \in \mathcal{H}_{NS} \}.
\end{align*}
\end{prop}
\begin{proof}
In the lattice $M$, we have $\left< e_1 - e_2 \right>^{\perp} 
= \left< 2 \right> \oplus \left< 6 \right> \oplus \left< -2 \right>^3 = T_{node}$. 
In fact, an orthogonal basis of $T_{node}$ is given by
\begin{align*}
  e_1 + e_2 + e_3,\qquad 3e_1 + 3e_2 - 3e_4 + e_5 + 2e_6, \qquad e_1 + e_2 + e_3 - e_4, \\ 
-e_1 - e_2 + e_4 - e_6, \qquad -e_1 - e_2 + e_4 - e_5 -e_6.
\end{align*}
Moreover we have 
\[
 z \perp (e_1 - e_2) \ \Leftrightarrow \ z_2 = 1 \ \Leftrightarrow \ 2 \det \Psi(z) = -1, \qquad
(z = [1:z_2:z_3:z_4:z_5:z_6] \in \DD_M^+).
\]
Similarly we have
\begin{align*}
\left< e_5 \right>^{\perp} &= \left< e_1, e_2, e_3, e_4, e_5 + 2e_6 \right>
= \mathrm{U} \oplus \mathrm{U}(2) \oplus \left< -12 \right> = T_{Eck}, \\
 \left< e_5 + 2 e_6 \right>^{\perp} &= \left< e_1, e_2, e_3, e_4, e_5 \right>
= \mathrm{U} \oplus \mathrm{U}(2) \oplus \left< -4 \right> = T_{NS} \\
\left< 3e_2 + e_5 + 2 e_6 \right>^{\perp} &= \left< e_2,\ 2e_1 + e_2 +e_5 + e_6,\ e_3,\ e_4,\ e_2 + e_5 \right>
= \mathrm{U}(2) \oplus \mathrm{U}(2) \oplus \left< -4 \right> = T_{Km}
\end{align*}
and
\begin{align*}
z \perp e_5 \ \Leftrightarrow \ z_6 = 2 z_5 \ \Leftrightarrow \ {}^t\Psi(z) = -\Psi(z), \qquad
z \perp (e_5 + 2 e_6) \ \Leftrightarrow \ z_6 = 0 \ \Leftrightarrow \ {}^t\Psi(z) = \Psi(z), \\
z \perp (3e_2 + e_5 + 2 e_6) \ \Leftrightarrow \ 2z_6 =1  \ \Leftrightarrow \ \Psi(z) = 
\begin{bmatrix} z_3 & z_5 + \frac{\w}{2} \\ z_5 + \frac{\w^2}{2} & z_4 \end{bmatrix}.
\end{align*}
\end{proof}
For a subgroup $G \subset \HG$ and a character $\chi$ of $G$, a modular form of weight $k$ with character $\chi$ is 
a holomorphic function on $\HH_2$ such taht
\[
 f((A \tau + B)(C \tau + D)^{-1}) = \chi(g) \det (C \tau + D)^k f(\tau) 
\]
for any $g = \begin{bmatrix} A & B \\ C & D \end{bmatrix} \in G$. Let $[G, k, \chi]$ be the vector space of such functions. 
In \cite{DeKr}, Dern and Krieg determined the graded ring $\oplus_{k=0}^{\infty} [\HG, k, \det^k]$. It is generated by 
Eisenstein series $E_k(\tau)$ of weight $k$ $(k=4,6,10,12)$, and Borcherds products $\phi_9(\tau)$ and $\phi_{45}(\tau)$  
vanishing exactly on $\HG$-orbits of $\mathcal{H}_{NS}$ and 
$\mathcal{H}_{Eck}$ respectively (They considered a lattice $\mathrm{U} \oplus$). By a linear map
\[
 [\HG, k, (\det)^k] \longrightarrow [\HG_0(2), k, (\det)^k], \quad f(\tau) \mapsto f(2 \tau),
\]
we obtain modular forms with respect to $\HG_0(2)$. 
\begin{prop} A modular form $\phi_9(2 \tau)$ vanishes on $\HG_0(2) \rtimes \left< T \right>$-orbit of 
$\mathcal{H}_{NS} \cup \mathcal{H}_{Km}$.
\end{prop}
\begin{proof}
Note that
\begin{align*}
 \phi_9(2\tau') = 0 \quad &\Leftrightarrow \quad 2 \tau'= (A\tau + B)(C\tau + D)^{-1} \\
&\Leftrightarrow \quad \tau' 
= \begin{bmatrix} A & \frac{1}{2}B \\ 2C & D\end{bmatrix} \cdot \frac{1}{2} \tau \qquad \qquad
(\begin{bmatrix} A & B \\ C & D\end{bmatrix} \in \HG, \ \tau \in \mathcal{H}_{NS}).
\end{align*}
Therefore the zero divisor of $\phi_9(2\tau)$ is the orbit of $\mathcal{H}_{NS}$ under the action of the group
\[
 G = \{\begin{bmatrix} A & \frac{1}{2}B \\ 2C & D\end{bmatrix} \ | \ 
\begin{bmatrix} A & B \\ C & D\end{bmatrix} \in \HG \}.
\]
By the same argument with Propostion \ref{unitary-prop}, we see that 
\[
 G = \bigcup_{i=1}^4 \HG_0(2) g^*(B_i), \qquad
B_1 = \begin{bmatrix} 1 & 0 \\ 0 & 0 \end{bmatrix}, \quad 
B_2 = \begin{bmatrix} 0 & 0 \\ 0 & 1 \end{bmatrix}, \quad
B_3 = \begin{bmatrix} 0 & \w \\ \w^2 & 0 \end{bmatrix}, \quad
B_4 = \begin{bmatrix} 0 & \w^2 \\ \w & 0 \end{bmatrix}.
\]
Now the proposition follows from the facts that
\[
 T \cdot \mathcal{H}_{NS} = g^*(B_i) \cdot \mathcal{H}_{NS} = \mathcal{H}_{NS} \quad (i =1,2), \quad 
g^*(B_3) \cdot \mathcal{H}_{NS} = \mathcal{H}_{Km}, \quad
g^*(B_4) \cdot \mathcal{H}_{NS} = T \cdot \mathcal{H}_{Km}.
\] 
\end{proof}
\section{Miscellaneous}
In \cite{FH}, Freitag and Hermann gave a degree two map $\HH_2 / \HG(2) \rightarrow X_{32} \subset \PP^5$ 
by thete functions, wehre $X_{32}$ is a $W(\mathrm{E}_6)$-invariant hypersurface of degree $32$ and
\[
 \HG(2) = \{ g \in \HG \ |\ g \equiv \I_2 \mod 2 \}.
\]
In \cite{K2}, Kond\={o} gave an embedding of the moduli 
space of Enriques surfaces with levele $2$ structure. The restriction of Kond\={o}'s map would give a similar map.
On the other hand, the group $W(\mathrm{E}_6)$ acts on the moduli space of ordered $6$  points on $\PP^2$. 
One may ask relation between
\[
 \{ \text{ordered} \ 6 \ \text{points on} \ \PP^2 \} \longrightarrow 
\{ \text{Hessian K3 surfaces} \}
\]
and $\HH_2 / \HG(2) \rightarrow \HH_2 / \HG_1(2)$. 
\\ \indent
The Hermitian modular variety $\HH_2 / \HG$ is considered as a moduli space of Abelian $4$-folds of 
Weil type. It is interesting to study the Kuga-Satake-Hodge correspondence for Hessian 
K3 surfaces. A geometric correspondence as in \cite{Pa} is desired.

\end{document}